\documentclass[openright,a4paper,american,11pt]{amsart}

\usepackage[latin1]{inputenc}

\usepackage{amsmath}
\usepackage{amssymb}

\usepackage{babel}

\usepackage{amsfonts}
\input xy
\xyoption{all}
\usepackage[dvips]{graphicx}
\usepackage{boxedminipage}
\usepackage{color}

\newcommand{\ZZ}{{\mathbb Z}}

\newcommand{\CC}{{\mathbb C}}
\newcommand{\RR}{{\mathbb R}}

\newcommand{\TT}{{\mathbb T}}

\renewcommand{\S}{{\mathcal S}}
\renewcommand{\P}{{\mathcal P}}
\renewcommand{\L}{{\mathcal L}}
\newcommand{\R}{{\mathcal R}}
\newcommand{\Q}{{\mathcal Q}}

\newcommand{\Ve}{\text{Vert}}
\newcommand{\Ed}{\text{Edge}}
\newcommand{\oS}{\overset{\circ}S}
\newcommand{\oC}{\overset{\circ}C}
\newcommand{\of}{\overset{\circ} f}

\newcommand{\oh}{\overset{\circ} h}

\newtheorem{thm}{Theorem}[section]
\newtheorem{defi}[thm]{Definition}
\newtheorem{prop}[thm]{Proposition}
\newtheorem{lemma}[thm]{Lemma}

          {\theoremstyle{definition}
}
          {\theoremstyle{definition}
\newtheorem{exa}[thm]{Example}}

\begin{document}
\title{Tropical Open Hurwitz  numbers}
\date{\today}
\author{Benoît Bertrand}
\address{I.U.T de Tarbes, Université Paul Sabatier, Institut Mathématiques de Toulouse, 
118 route de Narbonne, F-31062 Toulouse Cedex 9, France}
\email{benoit.bertrand@math.univ-toulouse.fr}
\author{Erwan Brugallé}
\address{Université Pierre et Marie Curie,  Paris 6, 4 place Jussieu, 75 005 Paris, France} 
\email{brugalle@math.jussieu.fr}
\author{Grigory Mikhalkin}
\address{Section de mathématiques
Université de Genève,
Villa Battelle, 7 route de Drize,
1227 Carouge, Suisse} 
\email{grigory.mikhalkin@unige.ch}

\begin{abstract}
We give a tropical interpretation of Hurwitz
numbers extending the one discovered in \cite{CJM}.
In addition we treat a generalization of Hurwitz numbers for
surfaces with boundary which we call open Hurwitz
numbers.\thanks{Research is supported in part by the project TROPGEO
  of the 
  European Research Council. 
Also B.B. is partially supported by the ANR-09-BLAN-0039-01, E.B. is
partially supported by the ANR-09-BLAN-0039-01 and ANR-09-JCJC-0097-01,  and
G.M. is partially supported by the Swiss National Science Foundation 
grants n° 125070 and 126817.}  
\end{abstract}

\maketitle

Hurwitz numbers are defined as the (weighted) number of ramified
coverings of a compact closed oriented surface $S$ of a given genus
having a given set of critical values with
 given ramification profiles. These numbers have a long history, and
have  connections to many areas of mathematics, among which we can mention
 algebraic geometry, topology, combinatorics,
and representation theory (see 
  \cite{ZvLa} for example).

Here we define a slight generalization of these numbers that we call
\textit{open Hurwitz numbers}. To do so, we fix not only points on $S$ and
ramification profiles, but  also a collection of disjoint
circles on $S$ and  the behavior of the coverings 
above each of these circles.
Note that the total space of the ramified coverings we consider now is allowed to have boundary components.

We also define tropical open Hurwitz numbers, and establish a
correspondence with their complex counterpart.
This can simply be seen as a translation in the tropical language
 of the computation
of open Hurwitz numbers
 by cutting $S$ along a collection of 
circles. 
A decomposition of  $S$ into pairs of pants 
reduces the problem to the enumeration of ramified coverings of
the sphere $S^2$ with 3 critical values. In the particular case where
all ramification points are simple, except maybe two of them, 
we recover the
tropical computation of double Hurwitz numbers in \cite{CJM}.

\vspace{1ex}
This note is motivated by the forthcoming paper \cite{Br9} where the
computation of
 genus 0 characteristic
numbers of $\CC P^2$ is reduced to  
enumeration of floor diagrams and computation of genus 0 open Hurwitz numbers.

\vspace{1ex}
We would like to 
thank Arne Buchholz and 
Hannah Markwig who pointed out an inaccuracy in the first version of the
discussion at
the end of the paper.

\section{Open Hurwitz numbers}\label{classical}

The data we need to define open Hurwitz numbers are
\begin{itemize}
\item $S$ an oriented connected closed compact surface;

\item $\L$  a finite
collection of disjoint smoothly embedded 
circles 
 in $S$; we denote by $\oS$ the surface $S\setminus
 \big(\bigcup_{L\in \L}
 L\big)$;

\item  $\P$ be a finite collection of points in $\oS$;

\item a number $\delta(S')\in\ZZ_{\ge 0}$ associated to each connected
  component $S'$ of $\oS$;  to each circle $L\in\L$ which is in the closure of the connected
  components $S'$ and $S''$ of $\oS$ (note that we may have $S'=S''$),
 we associate the number
  $\gamma(L)=|\delta(S')-\delta(S'')|$; 

\item a partition $\mu(p)$ of $\delta(S')$ associated to each point
  $p\in\P$, 
  where $S'$ is the connected component of $\oS$ containing $p$;

\item a partition $\mu(L)$ of $\gamma(L)$ associated to each circle
  $L\in\L$.

\end{itemize}


\vspace{2ex}

In this note we identify two continuous maps   $f : S_1\to S$ and $f':S_1'
\to S$ if  there exists
%
a homeomorphism  $\Phi:S_1\to S_1'$ such that  $f'\circ \Phi= f$.

Now let us denote by $\S$ the set of all
(equivalence class of) ramified coverings $f:S_1\to
S$ where

\begin{itemize}
\item $S_1$ is a  connected compact oriented surface with boundary;

\item $f(\partial S_1)\subset \cup_{L\in\L}L$;

\item $f$ is unramified over $S\setminus \P$;
 
\item  $f_{|f^{-1}(S')}$ has degree $\delta(S')$ for each connected
  component $S'$ of $\oS$;

\item  for each point $p\in\P$, if
  $\mu(p)=(\lambda_1,\ldots,\lambda_k)$, then $f^{-1}(p)$ contains
  exactly $k$ points, denoted by
  $q_1,\ldots,q_k$, and $f$ has ramification index 
$\lambda_i$ at $q_i$;

\item for each circle $L\in\L$, if
  $\mu(L)=(\lambda_1,\ldots,\lambda_k)$, then $f^{-1}(L)$ contains
  exactly $k$ boundary components of $S_1$, denoted by
  $c_1,\ldots,c_k$, and $f_{|c_i}:c_i\to L$ is an
 unramified covering of degree $\lambda_i$.

\end{itemize}

Note that the Riemann-Hurwitz formula gives us
$$\chi(S_1)=\sum_{S'}\delta(S')\left(\chi(S')- |\P\cap S'|  \right)  +
\sum_{p\in\P} l(\mu(p))$$
where $ l(\mu(p))$ is the length of the partition $\mu(p)$ (i.e. its cardinality as a multi-set of natural numbers).

\begin{defi}
The open Hurwitz number $H_S^\delta(\L,\P,\mu)$ is defined as
$$H_S^\delta(\L,\P,\mu)=\sum_{f\in  \S}\frac{1}{|Aut(f)|}  $$
where $Aut(f)$ is the set of automorphisms of $f$.
\end{defi}

\begin{figure}[h]
\centering
\begin{tabular}{ccc}
\includegraphics[width=4cm, angle=0]{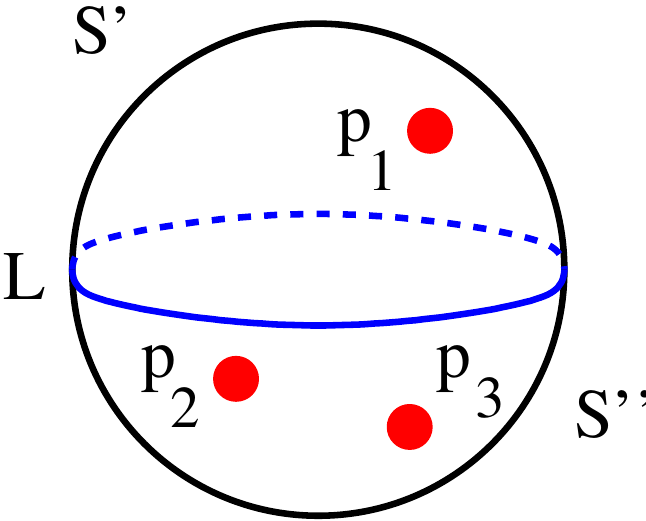}& \hspace{5ex} &
\includegraphics[width=4cm, angle=0]{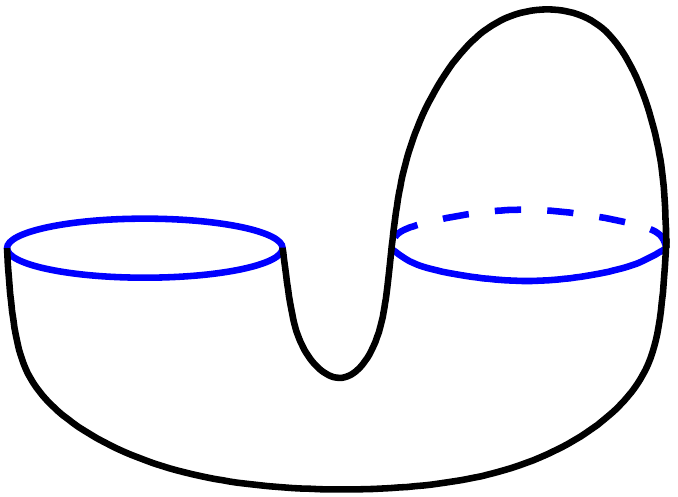}
\\
\\ a)  && b)  
\end{tabular}
\caption{}
\label{sphere}
\end{figure}

\begin{exa}\label{simple computation}
Let $S$ be the sphere, $L$ be a circle in $S$, and $p_1,p_2,$ and
$p_3$ three points distributed in $S$ as depicted in figure
\ref{sphere}a. Let us also denote by $S'$ and $S''$ the two connected
components of $S\setminus L$ according to figure \ref{sphere}a.
We define $\mu(p_1)=\mu(p_2)=\mu(p_3)=(2)$. 
The table below lists some values of $H_S^\delta(\L,\P,\mu)$ easily
computable by hand.
Figure \ref{sphere}b depicts the only map to be 
 taken into account
 in the second row of the table.
$$
\begin{array}{c|c|c|c|c|c||c}
\delta(S) & \delta(S') & \delta(S'') & \L & \mu(\L) & \P &  H_S^\delta(\L,\P,\mu)
\\\hline 2 & & & \emptyset & &\{p_2,p_3\} & \frac{1}{2}
\\\hline  &1 &2 & \{L\} & (1)&\{p_2,p_3\} & 1
\\\hline  &1 &2 &  \{L\}  & (1)   &\{p_1,p_3\} & 0
\\\hline  &0 &2 &  \{L\}  & (1,1) &\{p_2,p_3\} & \frac{1}{2}
\end{array}
$$
\end{exa}

\begin{exa}\label{simplify}
If $\mu(p)=(1,\ldots,1)$, then it is clear that 
$H_S^\delta(\L,\P,\mu)=H_S^\delta(\L,\P',\mu)$ where $\P'=\P\setminus
\{p\}$. 
\end{exa}

\begin{exa}\label{L to P}
Suppose that $L\in\L$ bounds  a disk $D$ 
 which contains no point of
$\P$, and that 
$\mu(L)=(\lambda_1,\ldots,\lambda_k)$ with $\lambda_i\ne 1$.
 Choose a point
 $p\in D$, 
define  
$\L'=\L\setminus
 \{L\}$,  $\P'=\P\cup \{p\}$, and 
extend $\mu$ at the point $p$ by
$\mu(p)=(1,\ldots,
 1,\lambda_1,\ldots,\lambda_k)$
where the number of 1 we add is equal to $\delta(D)$. Then 
 $H_S^\delta(\L,\P,\mu)=H_S^{\delta'}(\L',\P',\mu)$.
 Here $\delta'$ is obtained from $\delta$ by increasing it by
 $\gamma(L)$ over $D$. 

We have to assume that 
$\lambda_i\ne 1$ 
 for all $i$ in $\{1, \dotsc, k\}$
to get this identity. Otherwise,
new automorphisms of coverings might appear (e.g. the first two rows of
example \ref{simple computation}). 
\end{exa}

Note that 
the open Hurwitz number $H_S^\delta(\L,\P,\mu)$ is a topological
invariant that depend only on 
the topological type of the triple $(S,\oS,\P)$, 
and the functions $\delta$ and $\mu$.

\vspace{1ex}
In the special case where $\L$ is empty,
  we recover the usual Hurwitz numbers. In particular $\delta$ is
 just a positive integer number, the degree of the maps  we are
 counting and that  we denote by $d$. We simply denote Hurwitz numbers 
by $H_S^d(\P,\mu)$.

The problem of computing $H_S^d(\P,\mu)$ 
is equivalent to counting the number of some group morphisms from
the fundamental group of a punctured surface  to the
symmetric group $\mathfrak S_d$. Hence, Hurwitz numbers are theoretically
computed by Frobenius's Formula (see for example \cite[Appendix,
  Theorems A.1.9 and A.1.10]{ZvLa}).

\begin{exa}\label{elem1}
If $\P=\{p_1,p_2\}$ is a set of two points on the sphere $S^2$ with
$\mu(p_1)=\mu(p_2)=(d)$, then
$$H_{S^2}^d(\P,\mu)=\frac{1}{d}$$
\end{exa}

\begin{exa}\label{elem2}
If $\P=\{p_1,p_2,p_3\}$ is a set of three points on the sphere $S^2$ with
$\mu(p_1)=(\lambda_1,\lambda_2)$, $\mu(p_2)=(d)$, and
$\mu(p_3)=(2,1,\ldots,1)$, then
$$H_{S^2}^d(\P,\mu)=\frac{1}{|Aut(\mu(p_1))|}$$
\end{exa}

To end this section, let us mention the following nice
closed formula due to Hurwitz.

\begin{prop}[Hurwitz]
If  $\mu(p)=(2,1,\ldots,1)$ for all $p$ in $\P$ except for one point
$p_0$ for which we have
$\mu(p_0)=(\lambda_1,\ldots,\lambda_k)$,
then
$$H_{S^2}^d(\P,\mu)=\frac{d^{k-3}(d+k-2)!}{|Aut(\mu(p_0))|}\prod_{i=1}^k
\frac{\lambda_i^{\lambda_i}}{\lambda_i!}  $$
\end{prop}

\section{Tropical open Hurwitz numbers}

\subsection{Tropical curves with boundary}\label{defi trop curve}

Given a finite graph $C$ (i.e. $C$ has a finite number of edges and
vertices) we denote 
by $\Ve(C)$ the set of its vertices,
 by $\Ve^0(C)$ the set of its vertices which are not $1$-valent, 
 and by 
$\Ed(C)$ the set of its edges.

\begin{defi}
An irreducible  tropical curve $C$ with boundary is a connected 
  finite graph
with $\Ed(C)\ne\emptyset$
 such that 
\begin{itemize}
\item  $C\setminus  \Ve^\infty(C)$ is a complete metric graph for some
 set of 1-valent vertices  $\Ve^\infty(C)$ of $C$;

\item the vertices of $\Ve^0(C)$ have non-negative integer weights,
 i.e. 
$C$ is equipped  with a map
$$\begin{array}{ccc}
\Ve^0(C) &\longrightarrow & \ZZ_{\ge 0}
\\ v&\longmapsto & g_v
\end{array} $$

\item any $2$-valent vertex $v$ of $C$ satisfies $g_v\ge 1$.

\end{itemize} 

If  $v$ is an element of $\Ve^0(C)$,
 the integer $g_v$ is called the genus of $v$.
The genus of $C$  is defined as  
$$g(C)=b_1(C) + \sum_{v\in\Ve^0(C)}g_v$$
where $b_1(C)$ denotes the first Betti number of $C$. When
$g(C)=b_1(C)$, we say that the curve $C$ is explicit.

\vspace{1ex}
A boundary component of $C$ is a $1$-valent vertex which not in
$\Ve^\infty(C)$.
An element of $\Ve^\infty$ is called a leaf of $C$, and its adjacent edge is
called an end of $C$.
\end{defi}
By definition,  
the
 leaves of $C$ are 
at 
infinite distance from all the other
points of $C$. A tropical curve without any boundary component is said to be
\textit{closed}. Our definition of tropical curves with boundary extends the
definition of tropical curves with stops introduced by Nishinou in \cite{Nish1}.
We denote 
 by $\partial C$ the set of the boundary components of $C$, 
and by 
 $\Ed^{0}(C)$ the set of its edges which are not adjacent to a
1-valent vertex.

Since $O_1(\RR)=GL_1(\ZZ)$, 
 the data of the metric on $C$ is equivalent to the data of a
$\ZZ$-affine structure on each  edge of $C$, i.e. the data of a
 lattice $M_p\simeq\ZZ$ in each tangent line of such  an edge at
 the point $p$. 
In order to avoid unnecessary formal complications, we treat points on
edges of a tropical curve as 2-valent vertices of genus 0 
in the next definition.

\begin{defi}
A continuous map  $h : C_1\to C$ is a
(non-proper) tropical morphism between the two tropical curves $C_1$ and $C$
if
\begin{itemize}

\item   $h^{-1}(\partial C)\subset \partial C_1$;

\item for any edge $e$ of $C_1$, the set $h(e)$ is contained either in
  an edge of $C$ or in a
  vertex in $\Ve^0(C)$; moreover
 the restriction $h_{|e}$ is a dilatation by some
  integer $w_{h,e}\ge 0$ (i.e. $dh_p(M_p)=w_{h,e}M_{h(p)}$
  for any $p\in e$ in the first case, 
and $w_{h,e}$ is obviously  0 in the second case);

\item for any vertex $v$ in $\Ve^0(C_1)$, if we denote by 
  $e_1,\ldots,e_k$ the edges of $C$
  adjacent to 

  $h(v)$, and by  $e'_{i,1},\ldots,e'_{i,l_i}$ the 
edges of $C_1$ adjacent to $v$ such that $h(e'_{i,j})\subset e_i$,
  then one has the balancing condition
\begin{equation}\label{balancing}
\forall i,j, \quad \sum_{l=1}^{l_i}w_{h,e'_{i,l}}=
  \sum_{l=1}^{l_j}w_{h,e'_{j,l}}
\end{equation}
This number is called the local degree of
$h$ at $v$, and is denoted by $d_{h,v}$;

\item for any vertex $v$ in $\Ve^0(C_1)$,
  if $l$ (resp. $k$) denotes the number of edges $e$ of $C$
  (resp. of $C_1$ with $w_{f,e}>0$) adjacent to $h(v)$ (resp. to $v$) and $k>0$ then 
 one has the
  Riemann-Hurwitz condition
\begin{equation}\label{RH condition}
 k- d_{h,v}(2g_{h(v)} +l -2) +2g_v -2 \ge 0
\end{equation}
This number is 
denoted by $r_{h,v}$.
\end{itemize}

The morphism $h$ is called proper if $h^{-1}(\partial C)= \partial C_1$.

\end{defi}
Note that the definition implies that $h^{-1}(\Ve^\infty(C))\subset
\Ve^\infty(C_1)$ for a tropical morphism $h:C_1\to C$.
The Riemann-Hurwitz condition in the previous definition comes from
the classical Riemann-Hurwitz Theorem: if $S_1$ is a genus $g_v$
oriented surface with $k$ punctures, $S$ is a genus $g_{h(v)}$
oriented surface with $l$ punctures, 
and $f:S_1\to
S $ is a ramified covering 
  of degree $d_{h,v}$, then the left hand side of inequality (\ref{RH
    condition}) is the sum of the ramification index of all points of
  $S_1$. In particular, it is non-negative.

The integer $w_{h,e}$
 is
called the \textit{weight of the edge $e$ with respect to $h$}.  When
no confusion is possible, we will  speak 
about the weight of an edge, without referring to the morphism $h$.
If 
$w_{h,e}=0$, we say that the morphism $h$ \textit{contracts} the edge
$e$. 
The morphism $h$ is called
\textit{minimal} if
$h^{-1}(\Ve^\infty(C))= \Ve^\infty(C_1)$, i.e. $h$ does not contract
any end.

Two tropical morphisms $h:C_1\to C$ and $h':C_2\to C$ are said to be of
the same \textit{combinatorial type} if there exists a homeomorphism
of graphs $\phi:C_1\to C_2$  (i.e. we forget
about the metric on $C_1$ 
and $C_2$) such that $h=h'\circ \phi$,
$g(v)=g(\phi(v))$ for any vertex $v$ of $C_1$ 
 and 
$w_{h,e}=w_{h', \phi(e)}$ for all $e\in\Ed(C_1)$.

\begin{exa}\label{Example2}
We depicted in  figure \ref{ex morphism}a
 a tropical morphism form a
rational tropical curve with five leaves and one boundary component
to a rational curve with four ends. Three edges have weight 2 with
respect to $h$. 
In the picture of a tropical morphism $h:C_1\to C$, we do not precise the
lengths of edges of $C_1$ and $C$ since the length of edges
of $C$ and the weights of edges of $C_1$ determine the length of edges
of $C_1$.
\begin{figure}[h]
\centering
\begin{tabular}{cc}
\includegraphics[width=4cm, angle=0]{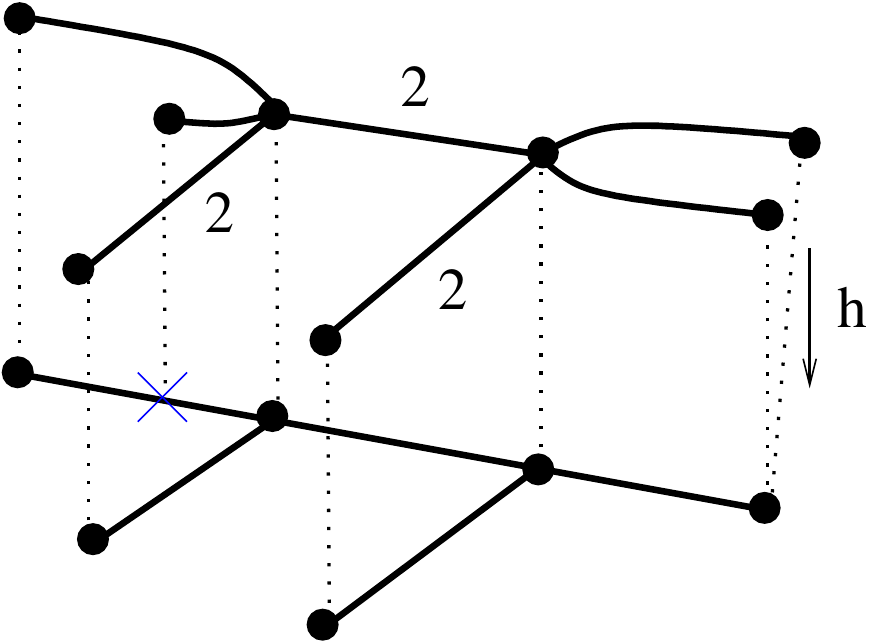}
&\includegraphics[width=4cm, angle=0]{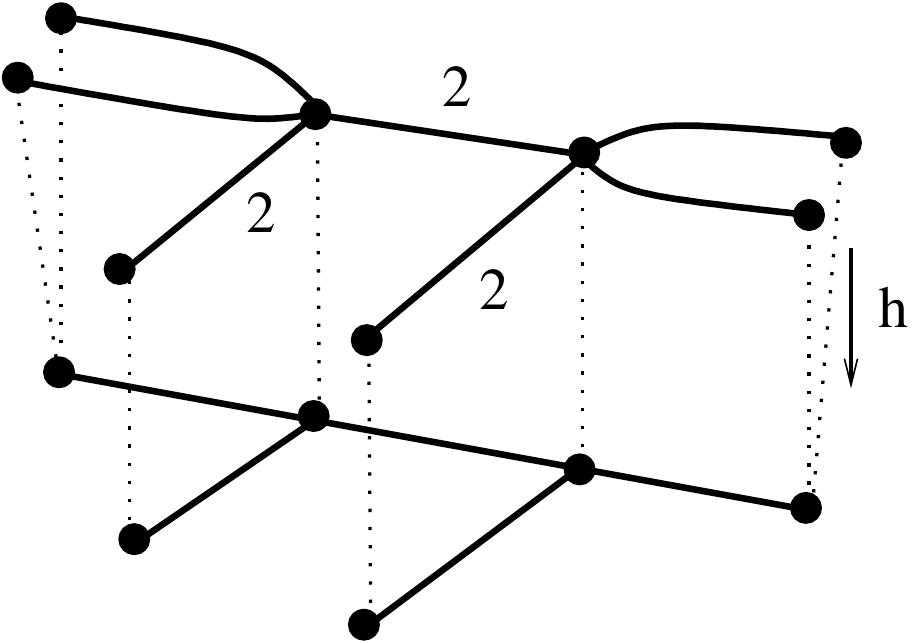}\\
a) & b)
\end{tabular}
\caption{Representation of two 
tropical 
morphisms.}
\label{ex morphism}
\end{figure}

\end{exa}

The sum of all local degrees of elements in $h^{-1}(v)$ is a locally
constant function on $C\setminus h(\partial C_1)$; if $C'$ is a
connected component of $C\setminus h(\partial C_1)$ then this sum
over a point of $C'$ is called the \textit{degree} of $h$ over $C'$.

Note that for the morphism from Figure \ref{ex morphism}~a) we  have two connected components of 
$C\setminus h(\partial C_1)$ as the boundary of $C_1$ consists of a single point (the only 1-valent
vertex whose image is inside an edge of $C$). The degree over the components of $C\setminus h(\partial C_1)$
are 1 and 2.

\begin{defi}
Let $h : C_1\to C$  be a tropical morphism.

A  subset $E$ of $C_1$ such that  $h(E)$ is a point of $C$ 
is called a ramification component of
$h$ if 
 $E$ is a connected component of $h^{-1}(h(E))$, and 
contains either an edge in $\Ed^0(C_1)$, or a vertex $v\in\Ve^0(C_1)$
with  $r_{h,v}>0$,
or a vertex $v\in\Ve^\infty(C_1)$
adjacent to an end $e$ with   $w_{h,e}>1$ (note that in such case $h(e)$ cannot
be a point as otherwise $w_{h,e}=0$).

If $p\in C$ is such that $h^{-1}(p)$ does not contain any ramification
component of $h$, we say that $h$ is unramified over $p$.

Let
$\nu=(\lambda_1,\ldots,\lambda_l)$ be an unordered $l$-tuple of
positive integer numbers.  
We say that the map $h$ has ramification
profile $\nu$ over the leaf $v$ of $C$ if $h^{-1}(v)=\{v_1,\ldots,v_l\}$
where $v_i$ is a leaf of $C_1$ adjacent to an end of weight $\lambda_i$.
\end{defi}

As in section \ref{classical}, we identify two tropical morphisms   
$h :C_1\to C$ and $h':C_1':\to C$ 
 if  there exist
two tropical isomorphisms $\Phi:C_1\to C_1'$ and 
  $\phi:C\to C$ such that
$\phi$ restricts to 
 the identity map on $\Q$ and $\R$, and $h'\circ \Phi=\phi\circ h$. 
An
automorphism $\phi$ of  a tropical morphism  $h : C_1\to C$ is
a tropical
isomorphism $\phi:C_1\to C_1$ such that
$h\circ\phi=h$.

\subsection{Definition of tropical open Hurwitz numbers}

Similarly to section  \ref{classical}, we start with the following data

\begin{itemize}
\item $C$ a closed explicit tropical curve with $\Ve^0(C)\ne \emptyset$;

\item $\R$  a finite
collection of points in $C\setminus\Ve(C)$ 
 such that any connected component of the set $C\setminus\R$, denoted
 $\oC$, contains a vertex of $C$;

\item  $\Q$ be a finite collection of points in $\Ve^{\infty}(C)$;

\item a number $\delta(C')\in\ZZ_{\ge 0}$ associated to each connected
  component $C'$ of $\oC$;  to each point $q\in\R$
 which is in the closure of the connected
  components $C'$ and $C''$ of $\oC$,
 we associate the number
  $\gamma(q)=|\delta(C')-\delta(C'')|$; 

\item a partition $\nu(q)$ of $\delta(C')$ associated to each point
  $q\in\Q$, 
  where $C'$ is the connected component of $\oC$ containing $q$;

\item a partition $\nu(q)$ of $\gamma(q)$ associated to each point
  $q\in\R$.

\end{itemize}

\vspace{2ex}

We denote by $\S^\TT$ the set of all minimal tropical morphisms $h:C_1\to
C$ such that
\begin{itemize}
\item $C_1$ is a tropical curve with boundary;

\item $h(\partial C_1)\subset \R$;

\item $h$ is unramified over $C\setminus \Q$;
 
\item $h_{|h^{-1}(C')}$ has degree $\delta(C')$ for each connected
  component $C'$ of $\oC$;

\item  for each point $q\in\Q$, the map $h$ has ramification profile $\nu(q)$
  over $q$;

\item for each point $q\in\R$, if
  $\nu(q)=(\lambda_1,\ldots,\lambda_k)$, the set $h^{-1}(q)$ contains
  exactly $k$ boundary components of $C_1$, denoted by
  $c_1,\ldots,c_k$, and $c_i$ is adjacent to an edge of $C_1$ of
  weight $ \lambda_i$.

\end{itemize}

Note that the fact that $h$ is minimal and unramified over $C\setminus
\Q$ implies that
$h(\Ve^0(C_1))\subset\Ve^0(C)$ and that $h$ does not contract any
edge of $C_1$. In particular the set $ \S^\TT$ is
finite. Moreover the length of edges of $C_1$ are completely determined by
the combinatorial type of $h$. In other words,
the following lemma 
holds. 
 
\begin{lemma}\label{unique}
Two distinct elements of $ \S^\TT$ have distinct combinatorial types.
\end{lemma}

As usual in tropical geometry, 
a tropical morphism
 $h:C_1\to C$   in $\S^\TT$ should be counted with some multiplicity.
Given
 $v$ a vertex in $\Ve^0(C_1)$ such that $h(v)$ 
is adjacent to the edges $e_1,\ldots, e_{k_v}$ of $C$, we choose a
configuration $\P'=\{p'_1,\ldots, p'_{k_v}\}$ of $k_v$ points on the sphere
$S^2$, and we define $\mu'(p'_i)$ as the 
partition of $d_{h,v}$ defined by $h$ at $v$ above the edge $e_i$
 (cf the balancing condition (\ref{balancing})). 

\begin{defi}\label{def-multiplicity}
The multiplicity of  $h:C_1\to C$  is
defined as
$$m(h)=\frac{1}{|Aut(h)|}\prod_{e\in\Ed^0(C_1)}w_{h,e}
\prod_{v\in\mathcal\Ve^0(C_1)}\  \left(\prod_{i=1}^{k_v} |Aut(\mu'(p'_i)|\right) \ 
H_{S^2}^{d_{h,v}}(\P',\mu')$$ 

The tropical open Hurwitz number $\TT H_C^\delta(\R,\Q,\nu)$ is defined as
$$\TT H_C^\delta(\R,\Q,\nu)=\sum_{h\in\S^\TT}m(h)$$ 
\end{defi}

As in section \ref{classical}, if $\R=\emptyset$ then $\delta$ is a
number denoted by $d$, and we denote by $\TT H_C^d(\Q,\nu)$ the
corresponding tropical (closed) Hurwitz number.

\begin{exa}\label{deg2-non-proper}
Let $h:C_1\to C$ be the tropical morphism depicted in figure 
\ref{ex morphism}a.
 It is the tropical analog of the map considered in figure~\ref{sphere}.
%
 Let $q_1$ be the image of the boundary component of $C_1$,
and $q_2$ and $q_3$ be the leaves of $C$ which are image of a leaf of
$C_1$ adjacent to an edge of weight 2. We denote by $C'$ (resp. $C''$)
the connected component of $C\setminus\{q_1\}$ which does not contain
(resp. contains) $q_2$ and $q_3$, and we define $\delta(C')=1$,
$\delta(C'')=2$, $\nu(q_1)=(1)$, and $\nu(q_2)=\nu(q_3)=(2)$.
 To compute $\TT H_C^\delta(\R,\Q,\nu)$, the morphism of figure~\ref{ex
   morphism}a is the only one to consider and it has multiplicity $1$
 so $\TT H_C^\delta(\R,\Q,\nu)=1$
(see the second row of the table in example~\ref{simple computation}).
\end{exa}

\begin{exa}
Let $C$ be a closed rational curve with four leaves.
 We set $\Q=\Ve^\infty(C)$ and $\nu(q)=(3)$ for $q\in\Q$.
Then according to figure \ref{deg 3}, we have 
$\TT H_C^3(\Q,\nu)=1$.

Indeed, for the first morphism of Figure \ref{deg 3} we have the (classical) Hurwitz numbers
at the inner vertices equal to $\frac{1}{3}$ each. In the same time the group of local automorphism
at each vertex is the symmetric group 
$\mathfrak S_3$,
 so its order is 6. Finally the group of automorphisms
of the morphism itself is also 
$\mathfrak S_3$.
Using Definition \ref{def-multiplicity} we get 
$\frac{1}{6}(\frac63)^2=\frac23$ as the multiplicity.

For the second morphism we have local Hurwitz numbers at the inner vertices
equal to $\frac13$ again. There are no local or global automorphisms, but there is an inner edge of weight 3.
Thus we get $3(\frac13)^2=\frac13$ as the multiplicity.

\begin{figure}[h]
\centering
\begin{tabular}{ccc}
\includegraphics[width=3cm, angle=0]{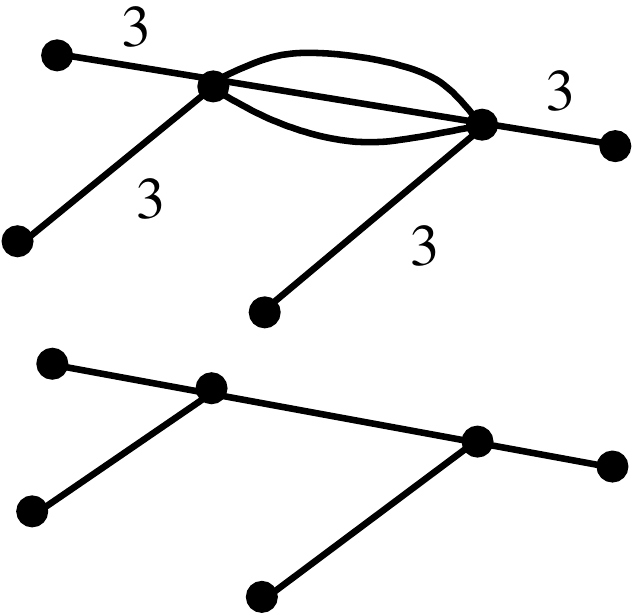}& \hspace{5ex} &
\includegraphics[width=3cm, angle=0]{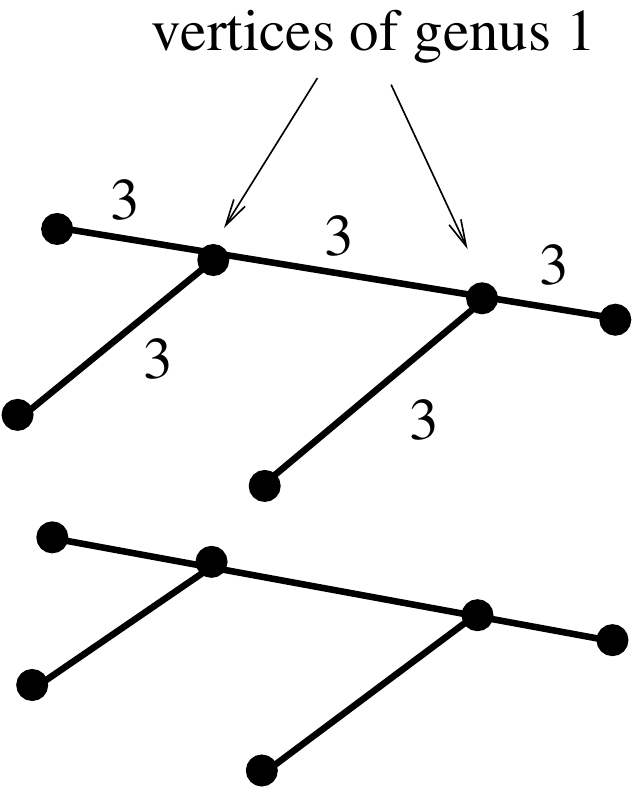}
\\
\\  $m(h)=\frac{2}{3}$  &&   $m(h)=\frac{1}{3}$ 
\end{tabular}
\caption{}
\label{deg 3}
\end{figure}

\end{exa}

\begin{exa}
Let $C$ be a genus 2 explicit tropical curve with two leaves and whose combinatorial type is as in figure \ref{gen 2}.
 We set $\Q=\Ve^\infty(C)$ and $\nu(q)=(2)$ for $q\in\Q$.
Then according to figure \ref{gen 2}, we have 
$\TT H_C^2(\Q,\nu)=8$.

\begin{figure}[h]
\centering
\begin{tabular}{ccc}
\includegraphics[width=4.5cm, angle=0]{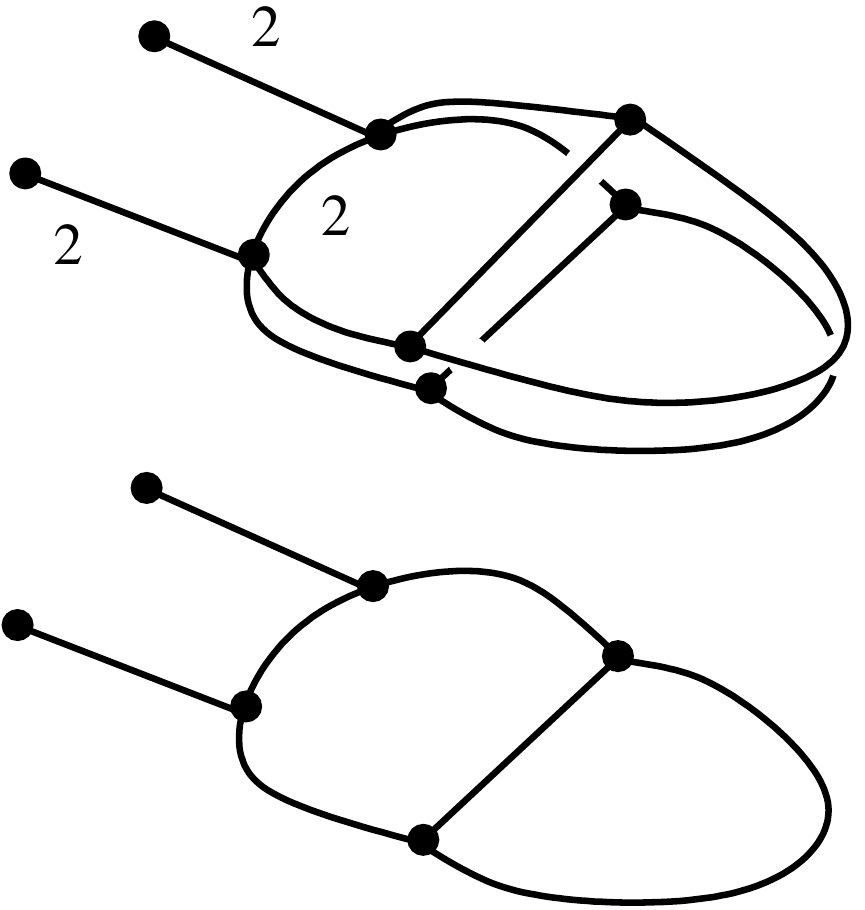}&
\includegraphics[width=4.5cm, angle=0]{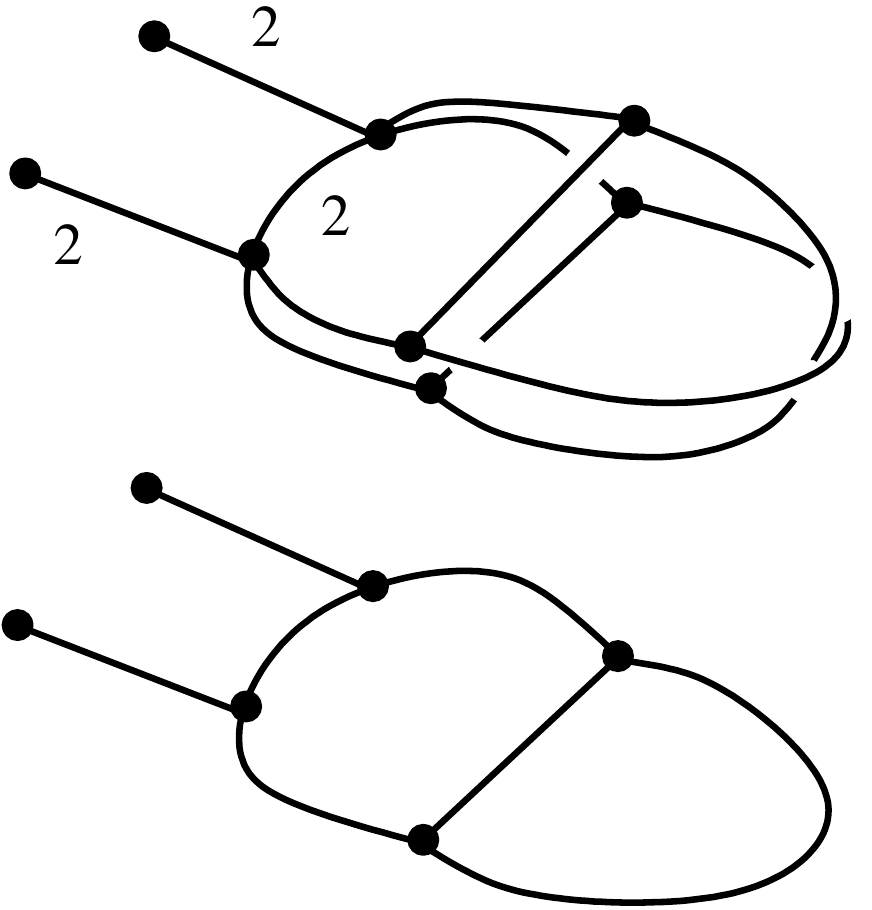}&
\includegraphics[width=4.5cm, angle=0]{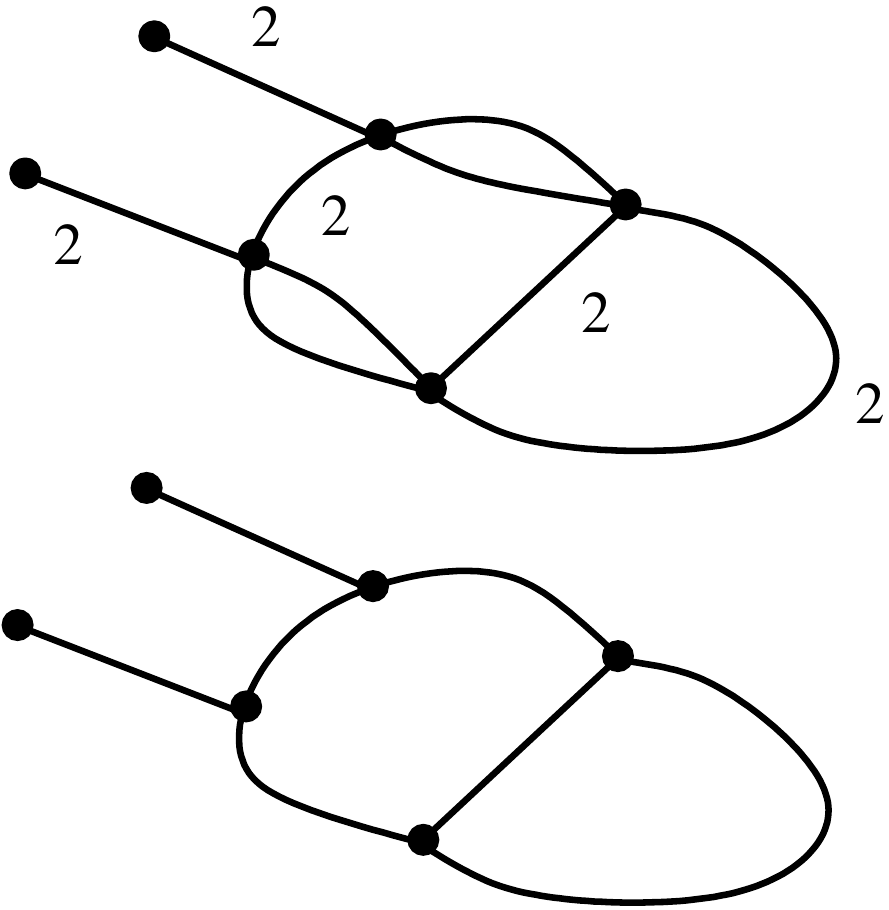}
\\
\\  $m(h)=1$  & $m(h)=1$  &  $m(h)=2$
\end{tabular}
\begin{tabular}{cc}
\includegraphics[width=4.5cm, angle=0]{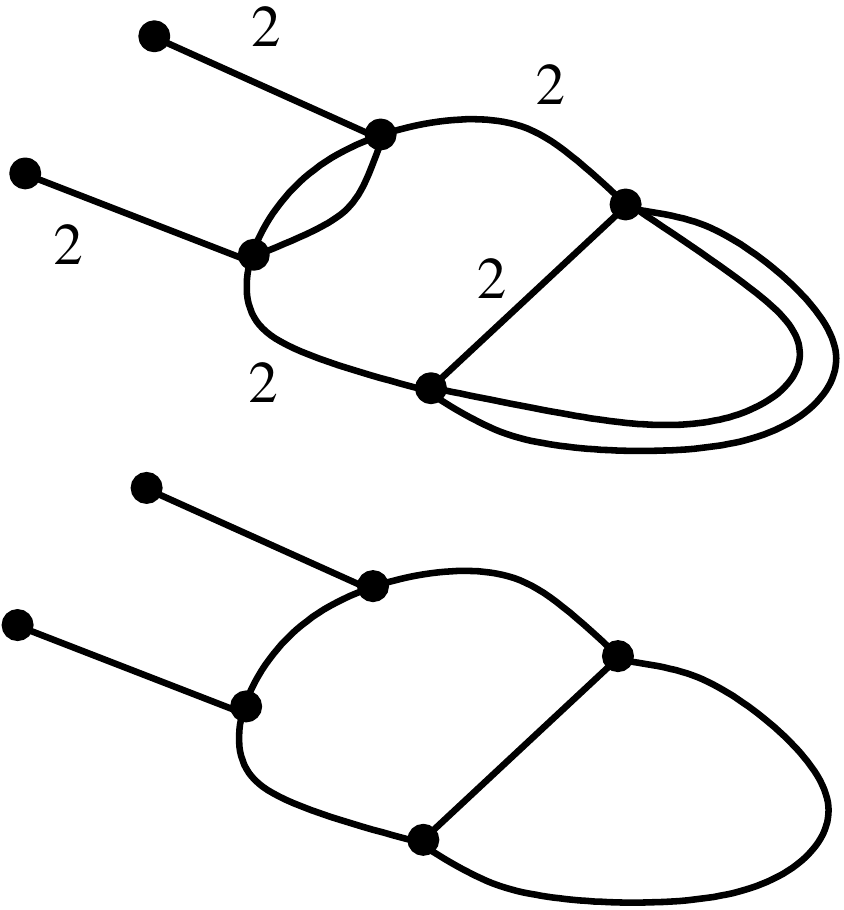}&
\includegraphics[width=4.5cm, angle=0]{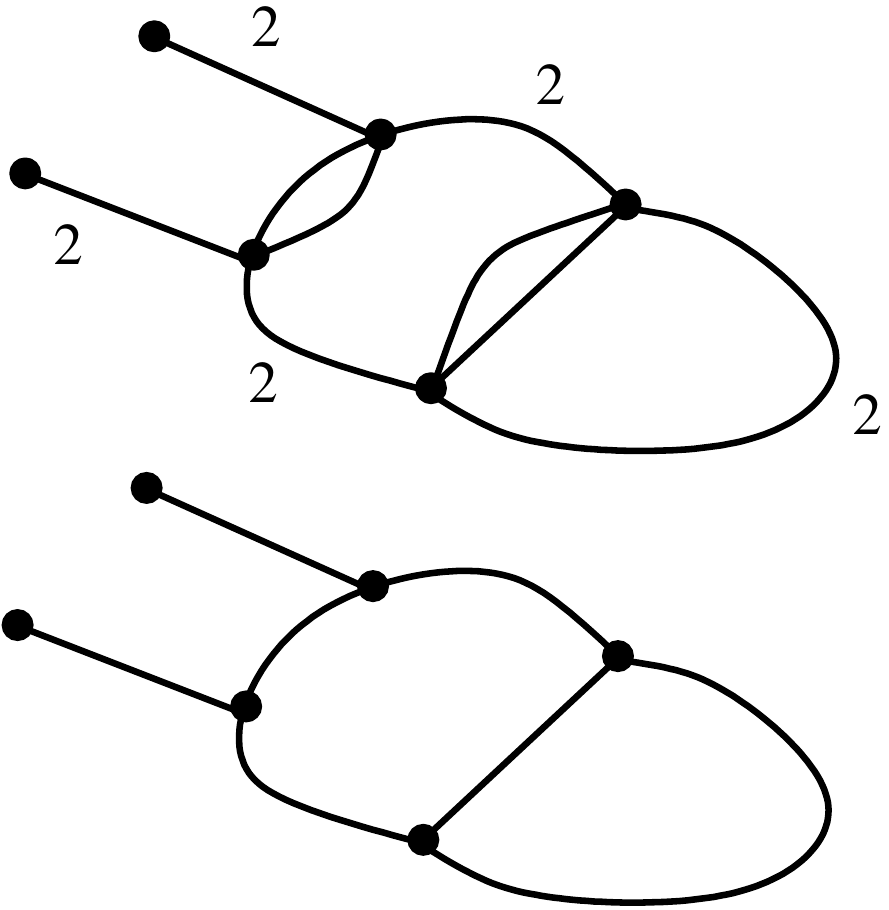}
\\
\\  $m(h)=2$ & $m(h)=2$
\end{tabular}
\caption{}
\label{gen 2}
\end{figure}

\end{exa}

\begin{exa}\label{L to P trop}
Suppose that $q\in\R$ is on an end with an adjacent leaf $q'$
 not in $\Q$, and that 
$\nu(q)=(\lambda_1,\ldots,\lambda_k)$ with $\lambda_i\ne 1$.
Define
 $\R'=\R\setminus
 \{q\}$,  $\Q'=\Q\cup \{q'\}$,
and  $\nu(q')=(1,\ldots,
 1,\lambda_1,\ldots,\lambda_k)$
where the number of 1 we add is equal to $\delta((qq'))$. Then 
 $H^\TT_\delta(C,\R,\Q,\nu)=H^\TT_{\delta'}(C,\R',\Q',\nu)$,
 where $\delta'$ is obtained from $\delta$ by increasing it by
 $\gamma(q)$ over $(qq')$.
This is the tropical counterpart of the identity described in
example~\ref{L to P}. 
Note that as in example \ref{L to P}, we have to assume that 
$\lambda_i\ne 1$ 
 for all $i$ in $\{1, \dotsc, k\}$ otherwise
new automorphisms of coverings might appear (e.g. 
Figures \ref{ex morphism}a and
b which are the tropical analogs of the morphisms corresponding to
 the first two rows of
example \ref{simple computation}). 
\end{exa}

\vspace{1ex}
Let us relate these tropical open Hurwitz numbers to the open Hurwitz
numbers we defined in section \ref{classical}.
Let $C$ be a tropical curve as in definition~\ref{def-multiplicity} with the data introduced at the beginning of this subsection.
Let $S$ be an oriented connected compact closed surface whose genus is
the genus of $C$.
We choose a collection $\L=\{L_q\}_{q\in\R}$ of disjoint
smoothly embedded circles in $S$ such that there is a natural
correspondence $C'\to S'_{C'}$ between the connected components of
$\oC$ and $\oS$ which preserves incidence relations and such that
$b_1(C')=g(C'_{S'})$. For each point $q\in\Q$, we choose a point
$p_q\in S'_{C'}$ where $C'$ is the connected component of $\oC$
containing $q$, such that $p_q\ne p_{q'}$ for $q\ne q'$ (see figure
\ref{dual graph} for an example). Finally we define
$\P=\cup_{q\in\Q}\{p_q\}$, $\delta(S'_{C'})=\delta(C')$,
$\mu(L_q)=\nu(q)$, and $\mu(p_q)=\nu(q)$.
 

\begin{figure}[h]
\centering
\begin{tabular}{c}
\includegraphics[width=10cm, angle=0]{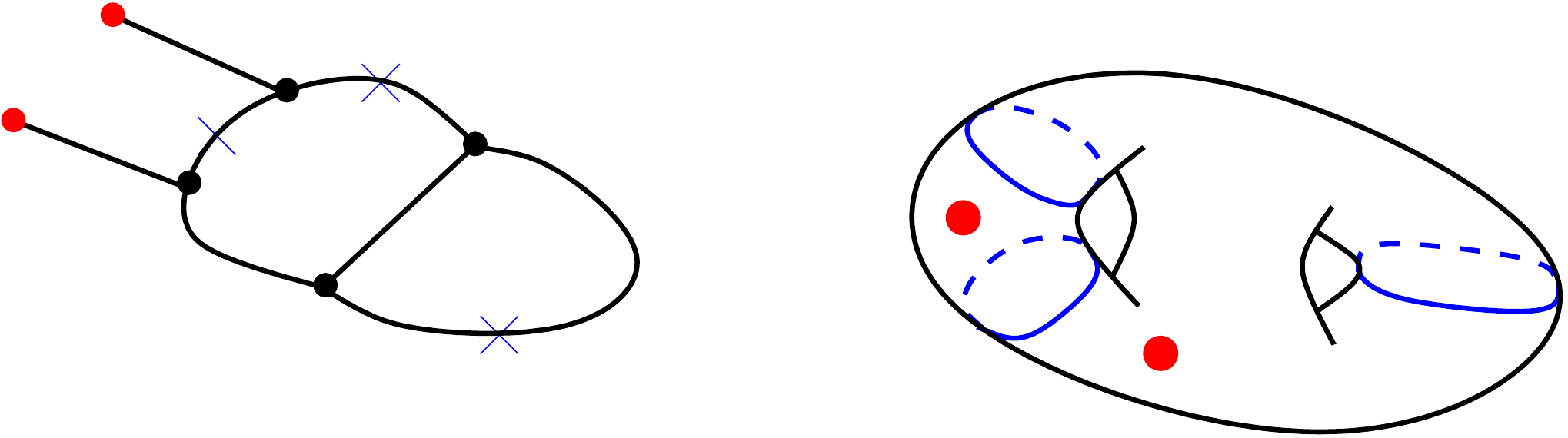}
\end{tabular}
\caption{A tropical curve $C$ is depicted on the left and the
  corresponding surface $S$ on the right. The two leaves of $C$ are
  elements of $\Q$ and correspond to elements of $\P$ (depicted by
  dots) on $S$ while crosses on $C$ represent points of $\R$ and
  correspond to the circles of $\L$ pictured on $S$.} 
\label{dual graph}
\end{figure}

\begin{thm}\label{tropH}
For any $\delta, \R, \Q$, and $\nu$, one has
$$\TT H_C^\delta(\R,\Q,\nu)=H_S^\delta(\L,\P,\mu)$$
\end{thm}
\begin{proof}
The proof is quite straightforward: there exists a natural surjection
$\phi:\S \to\S^\TT$, and for any element $h:C_1\to C$ we have 
$$m(h)=\sum_{f\in\phi^{-1}(h)}\frac{1}{|Aut(f)|}$$

Note that according to examples \ref{L to P} and \ref{L to P trop}, 
we may suppose that
$\Ve^\infty(C)=\Q=\P=\emptyset$.

\vspace{2ex}
Let us  construct the map $\phi$. 
For this we choose a configuration $\R'$ of points lying in the edges of
$C$ such that each edge in $\Ed^0(C)$ 
contains exactly one point of
$\R\cup\R'$, and we define $\oC_0=C\setminus\left(\R\cup\R'\right)$. 
For each point $q$ in $\R'$, we choose a smoothly embedded circle $L_q$
in $S$ such that $L_q\cap L_{q'}=\emptyset$ if $q\ne q'\in\R\cup\R'$,
 such that there is a natural correspondence $C'\to S'_{C'}$  between
the connected components of $\oC_0$ and 
$\oS_0=S\setminus \left(\L\cup_{q\in\R'}L_q\right)$ which
preserves 
incidence relations, and such that all connected components of $\oS_0$
have genus 0.
Note that  $C$ is a realisation of the dual graph of the decomposition
 of $S$ induced by $\left(\L\cup_{q\in\R'}L_q\right)$, and that
a connected component of $\oS_0$
 cannot be a disk since
 $\Ve^\infty(C)=\emptyset$.

Let $f:S_1\to S$ be an element of $\S$, and let us denote
$\oS_1=f^{-1}(\oS_0)$. 
We construct a graph $C_1$
 in the following way

\begin{itemize}
\item to each connected component $S_1'$ of
  $\oS_1$  corresponds  
a vertex $v_{S_1'}$ in $\Ve^0(C_1)$;
we set $g_{v_{S_1'}}$ to be equal to the genus of $S_1'$;

\item to each circle $L$ in 
$f^{-1}\left(\bigcup_{q\in\R\cup\R'} L_q\right)\setminus\partial S_1$
  adjacent to the connected components $S_1'$ and $S_1''$ of $\oS_1$
 corresponds an edge $e_L$
in $\Ed^0(C_1)$ 
joining  $v_{S_1'}$ and
$v_{S_1''}$; 

\item to each connected component $L$ of $\partial S_1$ adjacent to
the connected component $S_1'$ of $\oS_1$ 
 correspond a boundary point
  $v_L$ in $\partial C_1$ and an edge $e_L$ joining $v_L$ to $v_{S_1'}$.

\end{itemize}
For each circle  $L$ in 
$f^{-1}\left(\bigcup_{q\in\R\cup\R'} L_q\right)$, we denote by
$w(e_L)$  the degree of the unramified  
covering $f_{|L}:L\to f(L)$.
There exists a unique tropical
structure on $C_1$ and a unique tropical morphism $h:C_1\to C$ such
that 
$$h(v_{S_1'})=v \Leftrightarrow f(S_1')=S'_v,
\quad h(e_L)\subset e \Leftrightarrow f(L)=L_q \text{ where }
q=(\R\cup\R')\cap e,$$
and
$$
w_{h,e_L}=w(e_L)\quad \forall e_L\in\Ed(C_1)$$
 We define
$\phi(f)=h$.

\vspace{2ex}
Now it remains us to prove that
$m(h)=\sum_{f\in\phi^{-1}(h)}\frac{1}{|Aut(f)|}$ for any $h:C_1 \to C$ in $\S^\TT$.
For this, we reconstruct elements in $\phi^{-1}(h)$ step by step,
taking care of automorphisms.

The tropical curve $C_1$ induces a structure of reducible tropical
curve on
 the topological closure $\oC_1$ of $C_1\setminus
 h^{-1}(\R\cup\R')$. Note that to any vertex $v$ in $\Ve^0(C_1)$
 corresponds a connected component of $\oC_1$. The tropical morphism
 $h:C_1\to C$ induces a reducible tropical morphism $\oh:\oC_1\to C$.
 For each
 vertex $v$ in $\Ve^0(C_1)$ adjacent to the edges 
$e_1,\ldots,e_{k_v}$, we choose an unramified covering 
$f_{v}:S'_{1,v} \to S'_{f(v)}$ of degree $d_{h,v}$ where $S'_{1,v}$ is 
a surface with boundary components $L_1,\ldots,L_{k_v}$ of genus $g_v$ , and
$f_{v|L_i}$ is a degree $w_{f,e_i}$ unramified covering from $L_i$ to
$L_q$ where $q=(\R\cup\R')\cap e_i$. Let us set
$\oS_1=\cup_{v\in \Ve^0(C_1)}S'_{1,v}$, and let $\of:\oS_1 \to S$ be the map
defined by $\of_{|S'_{1,v}}=f_v$ for any $v$ in $\Ve^0(C_1)$.
The
union of all maps $\of$ constructed in this way
(up to equivalence) is a finite set $A_0$, and we clearly have
(in the notation of Definition \ref{def-multiplicity})
$$\sum_{\of\in A_0}\frac{1}{|Aut(\of)|}=\frac{1}{|Aut(\oh)|}
\prod_{v\in\mathcal\Ve^0(C_1)}\  \left(\prod_{i=1}^{k_v} |Aut(\mu'(p'_i)|\right) \ 
H_{S^2}^{d_{h,v}}(\P',\mu').$$ 

To get an element $f:S_1\to S$ of
$\S$, it 
 remains to glue all the 
 coverings $f_v$  according to the edges in $\Ed^0(C_1)$. 
Suppose we already have performed these gluings according to $s$ edges
in $\Ed^0(C_1)$ and obtained  a set $A_s$ of  coverings
of $S$,  and that we want now to glue elements of $A_s$ along the edge
$e$ in $\Ed^0(C_1)$.
Since we can identify in exactly $d$ different ways two degree $d$
unramified coverings of the circle by the circle, we have $w_{h,e}$
different ways to perform
 this gluing given an element $\widetilde f:\widetilde S_1\to S$ of
 $A_s$. 
We denote by 
$\widetilde f_1,\ldots,\widetilde f_{w_{h,e}}$ the coverings of $S$
constructed in this way.
Let $i=1,\ldots, {w_{h,e}}$. 
Any automorphism $\phi\in Aut(\widetilde f)$   fixing globally the
two boundary components of $\widetilde S_1$ corresponding to $e$
extends to an equivalence of coverings $\phi_i: \widetilde f_i\to
\widetilde f_j$ for some $j$. The homeomorphism $\phi_i$ is in $Aut(\widetilde
f_i)$ if $i=j$, and identify the two coverings $\widetilde f_i$ et
$\widetilde f_j$ otherwise. 
Hence at the end of this gluing procedure,
we obtain a set $A_{|\Ed^0(C_1)|}$ and we have
$$\sum_{f\in A_{|\Ed^0(C_1)|}}\frac{1}{|Aut(f)|}=m(h).$$ 
According to Lemma \ref{unique}, the set $A_{|\Ed^0(C_1)|}$
 is exactly $\phi^{-1}(h)$ so the Theorem is proved.
\end{proof}

\vspace{2ex}
We can allow points with ramification profile $\nu(q)=(2,1,\ldots, 1)$
in $C\setminus \Ve^\infty(C)$, and recover in this way  results from \cite{CJM}.
Suppose that $C$ is trivalent, and that there exists $q\in\Q$
with $\nu(q)=(2,1,\ldots, 1)$. Let $e$ be the end of $C$ adjacent to
$q$, and $v$ be the other vertex  adjacent to $e$.
We denote by $e_1,\ldots,e_k$ the edges of $C_1$ in $h^{-1}(e)$, and by
$\{v_1,\ldots,v_l\}=h^{-1}(v)$. By assumption on $\nu(q)$,  one of
the edges $e_j$, say $e_1$,  has weight 2 while the other edges $e_j$
have weight 1. Suppose that $e_1$ is adjacent to $v_1$. 
The
 Riemann-Hurwitz condition implies that 
all vertices $v_j$ are $(d_{h,v_j}+2)$-valent vertices of $C_1$, so the
vertex $v_1$ (resp. $v_j$, $j\ge 2$) has exactly 3 (resp. 2) 
adjacent edges not mapped to $e$.
Now according to Examples \ref{simplify}, \ref{elem1}, and \ref{elem2}
we have 
$$\prod_{j=1}^l \left(\prod_{i=1}^{3} |Aut(\mu'(p'_i)|\right) \ 
\TT H_C^{d_{h,v_j}}(\P',\mu')=(d_{h,v_1}-2)!\prod_{j=2}^l\frac{d_{h,v_j}!}{d_{h,v_j}}$$ 
The contraction $\pi:C\to C_0$  of $e$
 induces a contraction  $\pi_1:C_1\to C_2$ and a
tropical morphism $h':C_2\to C_0$ such that $\pi\circ h=h'\circ \pi_1$.
Moreover, all points in $h'^{-1}(\pi(q))$ 
lie in the interior of
edges of $C_2$, except one point which is a trivalent vertex of $C_2$ 
(see figure \ref{contraction}). 
We also have 
$$|Aut(h)|= |Aut(h')| (d_{h,v_1}-2)!\prod_{j=2}^ld_{h,v_j}!$$

So if $C$ is rational, $\L=\emptyset$, and all points  $p\in\P$ satisfy
$\nu(q)=(2,1,\ldots, 1)$ except maybe two of them, by contracting all
the points $q\in\Q$ with $\nu(q)=(2,1,\ldots, 1)$ we 
find ourselves in 
the
situation discussed  in \cite{CJM}. In particular, we recover the same
multiplicities in all cases except for example \ref{elem1} (in this
special case, the tropical computation from \cite{CJM} is wrong).

\begin{exa}
The result of this sequence of contractions in the case of the tropical
morphism depicted in figure
 \ref{ex morphism}~a) is depicted in figure
\ref{contraction}. 
\begin{figure}[h]
\centering
\begin{tabular}{c}
\includegraphics[width=4cm, angle=0]{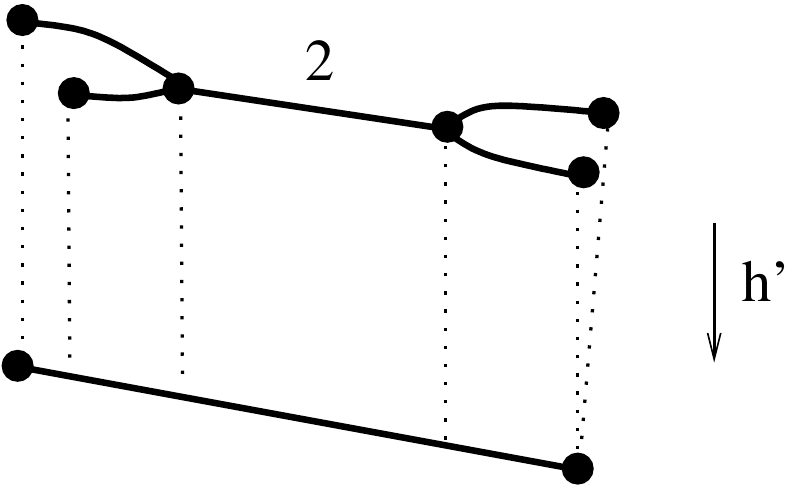}
\end{tabular}
\caption{}
\label{contraction}
\end{figure}
\end{exa}

\small
\def\rightmark{\em Bibliography}
\addcontentsline{toc}{section}{References}

\bibliographystyle{alpha}
\bibliography{../../Biblio.bib}

\end{document}